\documentclass[preprint,12pt,reqno, english]{svjour3}
\usepackage[T1]{fontenc}    
\usepackage{amsfonts}       
\usepackage{nicefrac}       
\usepackage{microtype}      
\usepackage{graphicx}
\usepackage{doi}
\usepackage{enumerate}
\usepackage[T1]{fontenc}
\usepackage[utf8]{inputenc}
\usepackage{amsmath}
\usepackage{amssymb}
\makeatletter
\newtheorem{thm}{Theorem}[section]
\newtheorem{defn}[thm]{Definition}
\newtheorem{prop}[thm]{Proposition}
\newtheorem{lem}[thm]{Lemma}
\newtheorem{rem}[thm]{Remark}
\newtheorem{cor}[thm]{Corollary}

\numberwithin{equation}{section}
\makeatother

\usepackage{babel}
\newcommand{\ldue}[1]{\ensuremath L^2([0,T],#1)}
\newcommand{\wto}{\ensuremath \rightharpoonup}
\newcommand{\palla}{\ensuremath \bar{B}_{R,m}}
\global\long\def\P{\mathbb{P}_{m}}%
\renewcommand{\Re}{\mathfrak{Re}}
\renewcommand{\d}[1]{\ensuremath \operatorname{d}#1}
\begin{document}

\title{Solutions to nonlocal evolution equations governed by non-autonomous forms and demicontinuous nonlinearities}
\titlerunning{Solutions to nonlocal evolution equations}

%
\author{Vittorio Colao        \and
        Luigi Muglia 
}


\institute{Department of Mathematics and Computer Science, UNICAL, Rende (CS), Italy\\
\email {(V. Colao) vittorio.colao@unical.it \and (L. Muglia) muglia@mat.unical.it}
}
\smartqed

\maketitle

\begin{abstract}
We deal with the existence of solutions having $L^2-$regulari\-ty for a class of  non-autonomous evolution equations. Associated with the equation, a general non-local condition is studied. The technique we used combines a finite dimensional reduction together with the Leray-Schauder continuation principle. This approach permits to consider a wide class of nonlinear terms by allowing demicontinuity assumptions on the nonlinearity.
\end{abstract}
\section{Introduction and preliminaries.}
In this paper we investigate the existence and regularity properties of the solutions to the evolution problem
 \begin{eqnarray}\label{mainproblem}
		\begin{cases}
			u'(t)+A(t)u(t)= f(t,u(t)), \qquad a.e. \ \ t\in[0,T]\ \\
			u(0)= g(u)
		\end{cases}
\end{eqnarray}
where $A(t)$ is a non-autonomous operator having $L^2-$maximal regularity, $f$ is nonlinear and demicontinuous  in the second variable and $g$ is a function representing a non-local condition.\\
Maximal regularity results have been the subject of several recent studies. On one side, this is due to the fact that the maximal regularity property of the solutions of differential equations gains a central role in the theory of parabolic problems, by permitting weaker requests on the regularity of the coefficients of differential operators. On the other hand, from an abstract point of view, maximal regularity characterizes a wide class of evolution equations and facilitates the application of linearization techniques.\\
This  property plays an essential part indeed, together with a finite dimensional reduction technique and the Leray-Schauder continuation principle, permit us to  prove the existence and regularity properties of the solutions to a wide class of non-autonomous semilinear evolution equations, including demicontinuous nonlinearities. \\
Turning our attention to the trace space, we should note that while in the autonomous case, the theory is well-developed and it is well known in the literature which spaces must be chosen in order to reach the desired regularity of the solutions, the non-autonomous setting had been the subject of several recent research papers.\\
The evolution problem is then completed by an initial condition of non-local type; it is remarkable that non-local conditions can be more useful than standard initial conditions $u(0)=u_0$ to model some physical phenomena (see \cite{KoPr,Nto} and references therein).\\
To be more precise, let $(V,\langle \cdot, \cdot \rangle_V)$ and $(H,\langle \cdot, \cdot \rangle_H)$ be two separable Hilbert spaces such that $V$ is continuously and densely embedded into $H$, i.e. $V$ is a dense subspace of $H$  such that
\[
	\| v\|_H \leq c \| v \|_V
\]
for some constant $c>0$ and any $v\in V.$
In many practical examples, an operator $\mathcal A (t)$ is associated to a bounded sesquilinear form $a(t,\cdot,\cdot)$ with domain $V$.\\
More precisely, assume that $a:[0,T]\times V \times V \to \mathbb{C}$ satisfies
\begin{itemize}
\item [(H1)] $a(\cdot,u,v)$ is strongly measurable for any $u,v\in V,$
\item [(H2)] there exists $M>0$ such that $|a(t,u,v)|\leq M \| u \|_V \|v\|_V$ for any $t\in [0,T]$ and $u,v\in V,$
\item [(H3)] there exists $\alpha>0$ such that $\Re( a(t,u,u) ) \geq \alpha \|u\|^2_V$ for any $t\in [0,T]$ and $u\in V;$
\end{itemize}
then, for any $t\in [0,T],$ $\mathcal{A}(t)\in\mathcal{L}(V,V')$ is well-defined by $\langle \mathcal{A}(t) u, v\rangle_{V' \times V} := a(t,u,v) $ and $\mathcal{D}(\mathcal{A}(t))=V,$ where $\langle \cdot , \cdot \rangle_{V' \times V}$ is the standard duality pairing.\\
Turning our attention to evolution problems governed by forms,
we mention the following result proved by J.L. Lions in 1961  (see \cite{Li}):

\begin{thm} For any fixed $x\in H$ and $f\in \ldue{V'},$ the problem
\begin{eqnarray}
		\begin{cases}
			u'(t)+\mathcal{A}(t)u(t)= f(t), \qquad a.e. \ \ t\in[0,T]\ \\
			u(0)=x,
		\end{cases}
\end{eqnarray}
has a unique solution $u\in\ldue{V}\cap  H^1([0,T],V').$
\end{thm}
We point out that the outstanding result by Lions only requires the measurability of $a(\cdot,u,v)$. On the other side, the above result is not fully satisfactory when applied to boundary value problems. Indeed, it is the part of $\mathcal{A}(t)$ which lies in $H$ that realizes the boundary conditions.\\
Therefore, let $A(t),\ t\in [0,T]$ be defined by $A(t) u:=\mathcal A (t) u,$ on the non-empty set $\mathcal{D}(A(t)):=\{u\in V: \mathcal A (t) u \in H\}$ and focus on
\begin{eqnarray}\label{PNA}
		\begin{cases}
			u'(t)+{A}(t)u(t)= f(t), \qquad a.e. \ \ t\in[0,T]\ \\
			u(0)=x.
		\end{cases}
\end{eqnarray}
\begin{defn}
For fixed $x$ in a suitable trace space, $\eqref{PNA}$ is said to have $L^2-$maximal regularity in $H$ if for any $f\in\ldue{H}$ there exists a unique solution $u\in H^1 ([0,T],H) \cap \ldue{V}$ and such that $u(t)\in \mathcal{D}(A(t)).$
\end{defn}
Several authors dealt with the problem of establishing which conditions on the trace space and which regularity assumptions on $a(t,u,v)$ are sufficient to achieve maximal regularity in $H.$ Among others, we cite the papers \cite{ACFP,A} and \cite{Pru}.\\
In 2015, Haak and Ouhabaz \cite{HaOu} gave a complete treatment of the evolution problem with non-zero initial values. We will mention their general result which we restate to our scope.\\
To this end, note that condition $(H3)$ readily implies that $A(t)$ is accretive for any $t\in [0,T],$ i.e.
\[ \Re \langle A(t) u, u \rangle_H \geq 0 \text{ for any } u\in \mathcal{D}(A(t)). \]  Secondly, assume that $a(t,u,v)$ also satisfies
\begin{itemize}
\item[(H4)] $|a(t,u,v)-a(s,u,v)|\leq \omega(|t-s|) \| u \|_V \| v \|_V,$ for some nondecreasing $\omega:[0,T]\to [0,+\infty)$ which satisfies
\[
\int_0^T  \frac{\omega(t)}{t^{3/2}} dt < \infty \text{ and the Dini condition } \int_0^T \left ( \frac{\omega(t)}{t} \right )^2 dt < \infty.
\]

\end{itemize}

\begin{thm}\cite[Corollary 3]{HaOu}\label{cor3}
Suppose that (H1)-(H4) and $\mathcal{D}(A(0)^{1/2})=V$ are satisfied, then for any $x\in V$ and any $f\in \ldue{H},$ \eqref{PNA} has $L^2-$maximal regularity in $H.$ Moreover there exists a constant $C_0>0$ such that
\begin{equation}\label{stimamaxreg}
\| u \|_{H^1 ([0,T],H)} + \| A(\cdot) u(\cdot) \|_{\ldue{H}} \leq C_0 \left ( \| f \|_{\ldue{H}} + \| x \|_V \right ).
\end{equation}
\end{thm}
To our scope we need to introduce a stronger condition on the domain, that is we assume that 
\begin{itemize}
	\item[(S)]  for any $t\in[0,T],$ the square root property $\mathcal{D}(A(t)^{1/2}) = V$  holds.
\end{itemize}
Three remarkable facts are detailed below. The first of them shows an immediate consequence of the above estimate, while the others two illustrate settings in which conditions (H3) and (S) can be lowered.
\begin{rem}\label{rem4} We stress that condition (H3) guarantees that the solution $u$ lies in $\ldue{V}.$  More precisely, there exists $C_1>0$ such that
\begin{eqnarray}\label{stimamaxreg2}
\nonumber \| u \|_{H^1 ([0,T],H)} + \| u \|_{\ldue{V}} \\
+ \| A(\cdot) u(\cdot) \|_{\ldue{H}} & \leq & C_1 \left ( \| f \|_{\ldue{H}} + \| x \|_V \right ).
\end{eqnarray}
Indeed, observe that
\begin{eqnarray*}
\alpha \| u(t) \|^2_V && \leq \Re \left ( a(t,u(t),u(t)) \right ) = \Re  \langle A(t) u(t), u(t) \rangle_H \\
\ && \leq \frac 1 2 \left ( \| u(t) \|^2_H + \|A(t)u(t)\|^2_H \right ),
\end{eqnarray*}
which implies, by means of \eqref{stimamaxreg} that
\begin{eqnarray*}
\|u \|^2_{\ldue{V}} &\leq& \frac 1 {2\alpha} \left ( \| u \|^2_{\ldue{H}} + \|A(t)u(t)\|^2_{\ldue{H}} \right )\\
&& \frac {C_0^2} {2\alpha} \left ( \| f \|_{\ldue{H}} + \| x \|_V \right )^2.
\end{eqnarray*}
Hence \eqref{stimamaxreg2} follows by setting $C_1:= C_0 \left ( 1 + \left ( \frac 1 {2\alpha} \right )^{\frac 1 2}\right ).$
\end{rem}
\begin{rem}\label{trucco}
	Let $\mu>0$ and set $v(t):=e^{-\mu t} u(t).$ Then, if $u$ satisfies problem \eqref{PNA}, then $v$ is the unique solution of
	\begin{eqnarray*}
		\begin{cases}
			v'(t)+{A}(t)v(t) + \mu v(t) =g(t), \qquad a.e. \ \ t\in[0,T]\ \\
			v(0)=u(0),
		\end{cases}
	\end{eqnarray*}

	where $g(t):=e^{-\mu t} f(t).$ \\ This shows that condition (H3) can be lowered by assuming that
	\begin{itemize}
		\item [(H3$^*$)] there exists $\alpha,\delta>0$ such that \[ \delta \| u \|^2_H + \Re( a(t,u,u) ) \geq \alpha \|u\|^2_V  \] holds for any $t\in [0,T]$ and $u\in V.$
	\end{itemize}
\end{rem}
\begin{rem}
	Assume that
	\begin{itemize}
		\item[(H4$^*$)]$|a(t,u,v)-a(s,u,v)|\leq \omega(|t-s|) \| u \|_V \| v \|_{V_\gamma}$
	\end{itemize}
	holds for some bounded $\omega$ and  where $V_\gamma$ is the complex interpolation space $[H,V]_\gamma$ for a fixed $\gamma\in (0,1).$ Then, it follows from \cite[Proposition 2.5]{ArMon} that property (S) can be relaxed by assuming
	\begin{itemize}
		\item[(S$^*$)]  $\mathcal{D}(A(t_0)^{1/2}) = V$ for some $t_0\in[0,T].$
	\end{itemize}
\end{rem}

Based on Theorem \ref{cor3} and by closely following the lines provided by \cite{ACFP}, in Section \ref{sez2} we construct a strongly continuous evolution family and introduce some basic properties which will be crucial in the sequel.\\
In several recent papers (see e.g. \cite{BeLoMaTa,LuLiOb,ZhLiXi} and references therein), a final dimensional reduction strategy had been proved to be a key point when associated to the study of autonomous evolution problems. In particular a new method has been introduced in  \cite{BLT} which combines Yosida approximation with the approximation solvability method. The aforementioned method is a generalization of the well known strong approximation technique (\cite{Bro,Pe}) and has been used in \cite{Loi} for studying periodic oscillations, while in \cite{BLMO} it has been applied to nonlocal differential equation problems in  Banach spaces. 
In this article, we show that the technique also fits well in the non-autonomous case and, as a novelty element, we combine the approximation solvability method with the approximation properties of evolution systems governed by forms. Section \ref{sez3} is devoted to the scope.\\ Indeed, we prove that Theorem  \ref{cor3} also applies to evolution  equations governed by operator of type $\P A(t),$ where $\P$ is the projection onto a finite dimensional space (Lemma \ref{Am}).  An approximation result is also introduced to show the uniform convergence of the evolution systems generated by $\P A(t)$ (Lemma \ref{lem:wmconverge}).\\
The above-mentioned approach permits us to lower the hypothesis on the nonlinear term, as explained below.\\
Let $X$ and $Y$ be Banach spaces and let  $f:[0,T]\times X \to Y$ be a map. An important concept in Functional Analysis is the one of superposition operator $N_f :L^p ([0,T],X) \to L^q([0,T],Y)$, defined by $N_f (u) (t) := f(t, u(t)).$
We recall the following classical theorem:
\begin{thm}[\cite{LuPa}]\label{LupA}
If $X$ and $Y$ are separable and $f$ is measurable in $[0,T]\times X,$ then $N_f:L^p ([0,T],X) \to L^q([0,T],Y)$ is well-defined if and only if
there exists a constant $a>0$ and a function $b\in L^q ([0,T],\mathbb{R}_+)$ such that
\begin{equation}\label{subl}
\| f(t,x) \|_Y \leq a \| x \|_X^{p/q} + b(t).
\end{equation}
Moreover $N_f$ maps bounded subsets into bounded subsets.
\end{thm}
In classical examples on $L^p$ spaces, nonlinear superposition operators are often constructed  on Caratheodory maps $f:[0,T]\times X \to Y$, i.e. by assuming that $f$ is continuous w.r.t. the second variable. The continuity of the superposition operator $N_f$ is then guaranteed as a consequence of this fact.\\
We stress that the continuity of superposition operator is almost a necessary property in the applications to differential systems. Nevertheless, the use of weak topologies can often be more convenient in several practical problems. Unfortunately, \cite[Example 2.3]{MoTe} shows that weak continuity of superposition operator $N_f$ acting on $L^p$ spaces does not derive from the same assumption on the map $f(t,\cdot).$ \\
Therefore a stronger form of continuity on $f(t,\cdot)$ appears to be necessary in order for the superposition operator to continuously act on spaces endowed with the weak topology (see \cite[Theorem 2.6]{MoTe}).\\
In Section \ref{sez4}, we show that the operator $N_f$ is demicontinuous  or, in other words, that for any sequence $x_n \to x$  it holds $N_f (x_n) \wto N_f(x),$ whenever $f(t,\cdot)$ satisfies the same property (Lemma \ref{weaktostrong}).\\
Demicontinuity fits well with the finite dimensional reduction approach and permit us, under the further assumption that the map $f$ satisfies a transversality condition, to apply the Leray-Schauder continuation principle (see also \cite{BeCi}):

\begin{thm}[\cite{LeSc}]\label{LeSc}
	Let $D$ be a closed and convex set in a Banach space $X.$ Let $S:[0,1]\times D \to X$  be a completely continuous operator (i.e. $S$ maps bounded subsets into compact subsets and it is continuous) and assume that
	\begin{enumerate}[(i)]
		\item $S(0,x)\in \operatorname{int} (D)$ for any $x\in D;$
		\item the set \[ \{ x \in  D : x = S(\lambda, x) \text{ for some } \lambda \in [0,1] \} \] is bounded and does not meet the boundary $\partial D$ of $D.$
	\end{enumerate}
	Then there exists $x\in D$ such that $x=S(1,x).$
\end{thm}
Finally, a solution to problem \eqref{mainproblem} is proved to exist as a limit. For this last step, the following result is crucial.
\begin{thm}[Aubin-Lions Lemma, see e.g. \cite{BoFa}]\label{aubinlions}
	Let $X_0,X$ and $X_1$ be three Banach spaces. Suppose that $X_0$ is compactly embedded in $X$ and $X$ is continuously embedded in $X_1$. Suppose also that $X_0$ and $X_1$ are reflexive. Then  for $0<T<+\infty$ and $1<r,s<\infty,$ we have that $L^r([0,T],X_0)\cap W^{1,s} ([0,T],X_1)$ is compactly embedded in $L^r([0,T],X).$
\end{thm}
The results detailed above are contained in Section \ref{sez5}.\\
The  paper is then concluded by Section \ref{sez6}, where application examples are discussed.  In particular,  weexploit the fact the demicontinuity assumption is naturally satisfied by any maximally monotone operator with sublinear growth.\\
We stress that the novelty elements that we introduce in this manuscript  lie in the approximation results (Lemmas \ref{Am} and \ref{lem:wmconverge}), in Lemma \ref{weaktostrong} where the demicontinuity of the superposition operator is shown and, leastly, in Theorem \ref{main} we show the existence of solutions in $H^1 ([0,T],H) \cap \ldue{V}$ to the problem
 \begin{eqnarray}\label{PMR}
	\begin{cases}
		u'(t)+A(t)u(t)=f(t,u(t)) \qquad t\in[0,T]\ \\
		u(0)=g(u).
	\end{cases}
\end{eqnarray}
It is worth mentioning that another novelty element of our approach lies in the regularity properties of the solutions to the nonlocal semilinear problem. Indeed, it is importanto to note that the majority of the papers on the topic mainly deal with the existence of mild solutions, that is continuous solutions of an associated fixed-point problem. Here we prove that the aforementioned approximation solvability method can be a key-strategy also for proving the existence of strong solutions. 
\section{Evolution families for operators governed by forms.}\label{sez2}
In this section we study the properties of the evolution family generated by a non-autonomous form which satisfies properties (H1)-(H4) and (S).\\
Here and through the rest of the paper,   $(V,\langle \cdot, \cdot \rangle_V)$ and $(H,\langle \cdot, \cdot \rangle_H)$ will be two separable Hilbert spaces, with $V$ densely embedded into $H$ and $\Delta$ will denote the set $\{(t,s): 0\leq s < t \leq T\}.$
As usual, $L^2([0,T],H)$ is the Lebesgue space of all the square integrable $H-$valued functions, while $H^1([0,T],H)=W^{1,2}([0,T],H)$ is the corresponding Sobolev space.\\
We turn our attention to the construction of a suitable evolution family by introducing the next lemma, a direct consequence of  Theorem \ref{cor3}, which proof follows the line of Propositions $2.3$ and $2.4$ and Corollaries $3.4$ and $3.5$ from \cite{ACFP}. Our hypotheses are slightly different and a proof is provided for sake of completeness.
 Also we use $C\in(0,+\infty)$ to denote any limiting constant, agreeing since now to redefine it whenever it is needed. 
\begin{lem}\label{evolsystem}
Assume that  (H1)-(H4) and (S) are satisfied, then there exists a contractive and strongly continuous evolution family $\{E(t,s)\}_{(t,s)\in \Delta} \subset \mathcal{L} (H)$ such that
\begin{itemize}
\item[(i)] $u(t):=E(t,s)x$ is the unique solution in $H^1 ([s,T],H) \cap L^2{([s,T],V)}$ of the homogeneous problem
\begin{eqnarray}\label{PA}
		\begin{cases}
			u'(t)+{A}(t)u(t)= 0, \qquad a.e. \ \ t\in[s,T]\ \\
			u(s)=x\in V,
		\end{cases}
\end{eqnarray}
moreover it holds
\begin{equation}\label{stimaevolV}
\| u \|_{H^1 ([s,T],H)} + \| u \|_{L^2([s,T],V)} + \| A(\cdot) u(\cdot) \|_{L^2([s,T],H)}   \leq C \|x\|_V
\end{equation}
where $C>0$ is a constant;
\item[(ii)] $u(t):=E(t,0)x$ is the unique solution in $H^1_\text{loc} ((0,T],H) \cap L^2_\text{loc}((0,T],V)\cap C([0,T],H)$ of the homogeneous problem  \begin{eqnarray*}
	\begin{cases}
		u'(t)+{A}(t)u(t)= 0, \qquad a.e. \ \ t\in[0,T]\ \\
		u(0)=x\in H,
	\end{cases}
\end{eqnarray*} and it holds
\begin{equation}\label{stimaevol}
 \|v\|_{H^{1}([0,T],H)}+\|v\|_{\ldue{V}}
 +\|A(\cdot)v(\cdot)\|_{\ldue{H}}  \leq C\sqrt{T}\|x\|_{H},
\end{equation}
where $v(t):=tE(t,0)x.$
\item[(iii)] for any $x\in H$ and $f\in \ldue{H},$ the function
\begin{equation}\label{PNAsol}
E(t,0)x+\int_0^t E(t,s) f(s) ds
\end{equation}
is the unique solution in $H^1_\text{loc} ((0,T],H) \cap L^2_\text{loc}((0,T],V)\cap C([0,T],H)$ of \eqref{PNA}.
\item [(iv)] for any $x\in V$ and $f\in \ldue{H},$ the unique solution $u$ of \eqref{PNA} given by Theorem \ref{cor3} has a continuous representation as
\[
u(t) = E(t,0)x+\int_0^t E(t,s) f(s) ds.
\]
\end{itemize}
\end{lem}
 \begin{proof}

From Theorem \ref{cor3} and by \cite[Proposition $2.3$]{ACFP}, a strongly continuous evolution family $\{E_0(t,s)\}_{(t,s)\in \Delta}\subset \mathcal{L}(V)$ exists such that $u(t):=E_0(t,s)x$ is the unique solution  in $H^1 ([s,T],H) \cap L^2([s,T],V)$ to the Cauchy problem \eqref{PA}.\\
Since $A(t)$ is accretive, it follows that
\begin{eqnarray*}
\| u(t) \|_H^2 - \|u(s)\|_H^2 && = \int_s^t \frac d {dt} \| u(\tau) \|_H^2 d\tau = 2 \int_s^t \Re \langle u'(\tau) , u(\tau) \rangle_H  d\tau \\
\ && = 2 \int _s^t - \Re  \langle A(\tau) u(\tau), u(\tau) \rangle_H  d\tau \leq 0,
\end{eqnarray*}
which readily implies that, for any fixed $(t,s)\in \Delta$, $$\| E_0(t,s) x \|_H = \|u(t)\|_H \leq \|u(s)\|_H = \| x \|_H.$$ This last inequality, together with the density of $V$ in $H$,  permit to uniquely extend $E_0(t,s):  V \to H$ to a linear operator  $E(t,s):H\to H$ by means of the BLT Theorem. The strong continuity of the family $\{ E(t,s) \}_{(t,s)\in \Delta}$ also follows from a density argument.\\

(i). The claims follow from above construction and (\ref{stimamaxreg2}) of Remark \ref{rem4}. \\

(ii).
Arguing as in \cite[Corollary $3.4$]{ACFP}, we start by fixing $x\in V.$ Set $v(t):= t E(t,0)x$ then, since $E(\cdot, 0)x\in \ldue{H}$ and by Theorem \ref{cor3}, $v$ is the unique solution of the non-homogeneous Cauchy problem
\begin{eqnarray*}
		\begin{cases}
			v'(t)+{A}(t)v(t)=E (t,0)x, \qquad a.e. \ \ t\in[0,T]\ \\
			v(0)=0,
		\end{cases}
\end{eqnarray*}
moreover by \eqref{stimamaxreg2},

\begin{eqnarray*}
\|v\|_{H^{1}([0,T],H)}+\|v\|_{\ldue{V}}
+\|A(\cdot)v(\cdot)\|_{\ldue{H}} &\leq& C\left(\|E(\cdot,0)x\|_{\ldue{H}}\right)\\
 &\leq& C\sqrt{T}\|x\|_{H}.
\end{eqnarray*}
 This last fact means that for any $\varepsilon \in (0,T),$
\begin{eqnarray}\label{stimaevoleps}
\nonumber \|E(\cdot,0)x\|_{L^2([\varepsilon,T],H)}+\|E(\cdot,0)x\|_{L^2([\varepsilon,T],V)}\\
\nonumber +\left\|\frac {dE (\cdot,0)x} {dt}\right\|_{L^2([\varepsilon,T],H)} & \leq & \frac {C\sqrt{T}} \varepsilon \| x \|_H.
\end{eqnarray}
  Again, by exploiting the density of $V$ in $H,$ it follows that the same inequalities also holds for any fixed $x\in H.$ We have then proved that $E(t,0)x$ lies in $H^1_\text{loc} ((0,T],H) \cap L^2_\text{loc}((0,T],V)\cap C([0,T],H)$ as well as the estimate (\ref{stimaevol}).\\

  (iii). It essentially follows as in \cite[Proposition $2.4$]{ACFP} for $p=2.$\\

  (iv). By maximal regularity, $u$ has a continuous representation in $C([0,T],H)$ (see \cite{Sh}) and $u$ solves \eqref{PNA} in $H^1_\text{loc} ((0,T],H) \cap L^2_\text{loc}((0,T],V)\cap C([0,T],H).$ Then result derives from the uniqueness of the solution.
 \end{proof}

\section{Finite dimensional reduction.}\label{sez3}

Let  $\{ \varphi_n\}_{n\in\mathbb{N}}$ be a Schauder basis for $V$ and hence for $H$ since $V$ densely embeds onto $H$. We assume  $\{ \varphi_n\}_{n\in\mathbb{N}}$ to be orthogonal w.r.t. the inner product of $H,$ $\langle \cdot, \cdot \rangle_H .$ For a fixed $m\in\mathbb{N},$ let $\P:H \to \operatorname{span_{\mathbb{C}}}\{ \varphi_j : 1\leq j\leq m\}$ be the projection $\P (\sum_{j=1}^\infty v_j \varphi_j) := \sum_{j=1}^m v_j \varphi_j.$ To avoid confusion, by $H_m$ we will denote the space $\P H$ endowed with the norm $\| \cdot \|_H,$ while $V_m$ indicates that the equivalent norm $\| \cdot \|_V$ is considered on the same finite dimensional vector space.\\
Note that $\P$ is self-adjoint with respect the $H$-inner product (that is  w.r.t. $\langle \cdot, \cdot \rangle_H$ ), indeed whenever $u,v \in H,$ we have
\[
\langle u, \P v \rangle_H = \lim_{k\to \infty} \langle \sum_{i=1}^k u_i \varphi_i, \sum_{j=1}^m v_j \varphi_j  \rangle_H =  \langle \sum_{i=1}^m u_i \varphi_i, \sum_{j=1}^m v_j \varphi_j  \rangle_H =\langle \P u, v\rangle_H.
\]
Of course, since $\{ \varphi_n\}_{n\in\mathbb{N}}$ is a Schauder basis for both $V$ and $H,$ it holds that $\| v - \P v\|_V \to 0$ and $\|x-\P x \|_H \to 0$ as $m\to\infty$ for any fixed $v\in V$ and $x\in H.$\\
Let $a:[0,T] \times V\times V \to \mathbb{C}$ be a sesquilinear form satisfying (H1)-(H4). For fixed $m\in\mathbb{N}$ we denote with $a_m : [0,T] \times V\times V \to \mathbb{C}$ the sesquilinear form given by
\[
a_m (t,u.v) := a(t,\P u, \P v) + \alpha \langle (I-\P)u , (I-\P)v \rangle_V.
\]
\begin{rem}
Note that $a_m$ also satisfies properties (H1)-(H4). Indeed, (H1) and (H4) are trivially satisfied while (H2) is easily derived from
\begin{eqnarray*}
&&| a_m (t,u,v) | \leq M \| \P u \|_V \|\P v\|_V + \alpha \| (I - \P) u \|_V \|(I-\P) v\|_V \\
&&\leq k_m (M+\alpha )\|u\|_V \|v\|_V,
\end{eqnarray*}
since both $\| \P \|_{\mathcal{L}(V)}$ and $\| I-\P \|_{\mathcal{L}(V)}$  are bounded by some $k_m >0$ for any fixed $m\in \mathbb{N}.$\\
To prove  (H3) observe that
\begin{eqnarray*}
&& \Re ( a_m(t,u,u) ) = \Re ( a (t, \P u, \P u) )+ \alpha \| I - \P u \|_V^2 \\
&& \geq  \alpha \| \P u \|_V^2 +  \alpha \| I - \P u \|_V^2\\
&& = 2 \alpha (\frac{1}{2}  \| \P u \|_V^2 + \frac{1}{2} \| I - \P u \|_V^2 )\\
&& \geq  2 \alpha\left \| \frac{1}{2} \P u + \frac{1}{2} u -\frac{1}{2} \P u\right\|_V^2\\
&& = \frac{\alpha}{2} \| u \|_V^2
\end{eqnarray*}
since $\| \cdot \|_V^2$ is convex.\\
\end{rem}
Let $A_m (t)$ be the operator defined by
\[
a_m(t,u,v) := \langle A_m (t) u, v \rangle_H \text{ for all }v\in V
\]
and with domain $\mathcal{D}(A_m (t)):=\{ u \in V: A_m (t) u \in H\}.$
\begin{rem}
Let $B:V \to H$ be the operator associated to the sesquilinear form $\langle \cdot , \cdot \rangle_V.$ Then for any fixed $t\in[0,T]$ and for any $u,v\in V$ one has
\begin{eqnarray*}
&& \langle A_m (t) u, v\rangle_H = a_m(t,u,v) = a(t,\P u, \P v) + \alpha \langle (I-\P) u, (I-\P) v \rangle_V \\
&&= \langle A (t)\P u, \P  v \rangle_H + \alpha \langle B (I-\P) u, (I-\P) v \rangle_H \\
&& =  \langle (\P A(t) \P +\alpha  \langle (I-\P) B (I-\P))u,v\rangle_H.
\end{eqnarray*}
Therefore, the identity
\begin{equation}\label{id}
A_m (t) = \P A(t) \P + (I-\P)B(I-\P)
\end{equation}
follows by density. It is then straightforward to see that  for any $u\in V_m,$
\[
A_m(t) u = \P A(t) u
\]
holds.
\end{rem}
\begin{lem}\label{Am} For any fixed $m\in\mathbb{N},$ $A_m(t)$ defined above generates a contractive evolution system $\{E_m (t,s) \} \subset \mathcal{L}(V)$ such that for any $x\in V,$ $u(t):=E_m(t,s)x$ is the unique solution in $H^1 (s,T,H) \cap L^2(s,T,V)$ of the homogeneous problem
\begin{eqnarray}\label{PAm}
		\begin{cases}
			u'(t)+{A_m}(t)u(t)= 0, \qquad a.e. \ \ t\in[s,T]\ \\
			u(s)=x \in V,
		\end{cases}
\end{eqnarray}
which satisfies inequality \eqref{stimaevolV}. Moreover, if $x\in V_m$ then $u(t)\in V_m$ for any $t\in [0,T].$
\end{lem}
\begin{proof}
The existence of the evolution system $\{ E_m (t,s) \}$ will directly follow from Lemma \ref{evolsystem} once we prove that $A_m(t)$ satisfies property (S). \\
Fix $t\in[0,T].$ It follows from \cite[Corollary 7.1.4]{arendtvectorvalued} that the operator $A_m(t) = \P A(t) \P + (I-\P) B (I-\P)$ generates a cosine function on $H,$ since $\P A(t) \P$ is bounded and $(I-\P)B(I-\P)$ is symmetric.\\
 By \cite[Corollary 5.18]{haase2003}, the numerical range $W(A_m(t)):=\{ \langle A_m(t) u, u \rangle_H | u\in V, \|u \|_H=1\}$ is then contained in a parabola.
   Lastly, by \cite[Theorems A and C]{mcintosh82} it follows that the square root condition $\mathcal{D}(A_m(t)^{1/2} )=V$ holds (see also \cite{arendt2004}).  This proves the first part.\\
Let $x\in V_m$ and let $u(t)$ be the solution of \eqref{PAm}. Note that the function $t\to u(t)-\P u(t)$ has a.e. derivative $ u'(t) -\P u'(t)$ since $\P$ is linear; moreover it holds
\begin{eqnarray}\label{pupu}
\nonumber&& \frac{1}{2} \frac{d}{dt} \| u(t) - \P u(t) \|_H^2 = \Re\langle u' (t) - \P u' (t), u(t)- \P u(t) \rangle_H \\
\nonumber && = - \Re\langle A_m (t)  u(t), u(t)-\P u(t) \rangle_H +  \Re \langle \P A_m (t)  u(t), u(t)-\P u(t) \rangle_H\\
\nonumber&& = - \Re\langle A_m (t)  u(t), u(t)-\P u(t) \rangle_H + \Re \langle A_m (t)  u(t), \P ( u(t)-\P u(t) ) \rangle_H \\
\nonumber && = -\Re \langle A_m (t)  u(t), u(t)-\P u(t) \rangle_H \\
\nonumber && = - a_m (t, u (t), u(t) - \P u (t) )\\
&& = -a (t, \P u(t), \P ( I-\P) u(t)) - \alpha \| (I-\P) u(t) \|_V^2,
\end{eqnarray}
where we have used the fact that $\P$ is self-adjoint w.r.t $\langle \cdot, \cdot \rangle_H$ and $I-\P$ is idempotent. Note that $\P (I-\P)  = 0,$ which implies that $a (t, \P u(t), \P ( I-\P) u(t))=0$ and \eqref{pupu} brings to
\[
\frac{d} {dt}  \| u(t) - \P u(t) \|_H^2 \leq 0 - \alpha \| (I-\P) u(t) \|_V^2 \leq 0.
\]
The direct consequence of this last is that for a.e. $t\in [0,T],$ $ \| u(t) - \P u(t)\|_H \leq \| x - \P x\|_H=0$ since $x\in V_m;$ thus $u(t)=\P u(t)$ and $u(t)\in V_m$ follows.

\end{proof}

\begin{lem}\label{lem:wmconverge}Let $\{E(t,s)\}_{(t,s)\in \Delta}$ and $\{E_{m}(t,s)\}_{(t,s)\in \Delta}$ be
the evolution systems generated by $A(t)$ and $A_{m}(t)$ respectively.
Then for  any fixed $x\in H,$ $\{ E_{m}(t,s)\P x\}_{n\in\mathbb{N}}$  converges in $H$ to $E(t,s) x$ uniformly on $t>s$ in $[0,T].$
\end{lem}
\begin{proof}
We start by fixing $x\in V$ and, for fixed $(t,s)\in\Delta,$ set $u(t):=E(t,s)x,$ $u_m(t):=E_m(t,s)\P x$ and $z_m(t):=u_m(t)-u(t).$  Note that by definition,
\begin{eqnarray}\label{bella}
\nonumber &&\frac{1}{2} \frac{d}{dt} \| z_m (t) \|_H ^2 \\
\nonumber && = -\Re \langle \ \P A (t) u_m (t), u_m (t) - u(t) \rangle_H +\Re \langle A (t) u(t), u_m (t)-u(t) \rangle_H \\
\nonumber &&= -\Re \langle  A (t)  u_m (t), \P u_m (t) - \P u(t) \rangle_H +\Re \langle A (t) u(t), u_m (t)-u(t) \rangle_H \\
\nonumber &&=  -\Re \langle A (t)  u_m (t), u_m (t) -  u(t) \rangle_H + \Re \langle  A (t)  u_m (t),\P u (t) -  u(t) \rangle_H\\
\nonumber &&\phantom{=}+\Re \langle A (t) u(t), u_m (t)-u(t) \rangle_H \\
\nonumber && \leq\Re \langle  A (t)  u_m (t),\P u (t) -  u(t) \rangle_H \\
&& = a(t,u_m (t),  \P u (t) -  u(t)) \leq  M \| u_m (t) \|_V \|u(t)-\P u(t) \|_V,
\end{eqnarray}
since $A$ is accretive.
By integrating, from the last and by H\"{o}lder inequality, one gets
\begin{eqnarray}\label{maggior}
\nonumber && \| z_m (t) \|_H^2 - \| z_m(s) \|_H^2 \leq 2M \| u_m \|_{L^2 ([s,T],V)} \|u-\P u \|_{L^2 ([s,T],V)}\\
&&\leq 2M \|  \P x \|_H  \|u-\P u \|_{L^2 ([s,T],V)}.
\end{eqnarray}
Fix  $\varepsilon>0$ and let $u_\varepsilon$ a  continuous $\varepsilon/4-$ approximation of $u$ in ${L^2 ([s,T],V)}.$ Let $\tau_\varepsilon$ be such that
\[
T \|(I - \P) u_{\varepsilon} (\tau_\varepsilon)\|_V^2  = \int_s^T \|  (I - \P)  u_{\varepsilon} (t)  \|_V^2 dt=\| (I- \P) u_{\varepsilon} \|^2_{L^2 ([s,T],V)}
\]
and let $m$ be big enough so that $T \|(I - \P) u_{\varepsilon} (\tau_\varepsilon)\|_V^2<\varepsilon^2/4.$ We immediately derive
\begin{eqnarray*}
&&\|u-\P u \|_{L^2 ([s,T],V)} \leq 2\| u - u_\varepsilon \|_{L^2 ([s,T],V)} + \| (I- \P) u_{\varepsilon} \|_{L^2 ([s,T],V)} \leq \varepsilon.
\end{eqnarray*}
Note that $\| z_m(s) \|_H =\| (I-\P)x\|_H \to 0$ as $m\to \infty$ since the embedding of $V$ into $H$ is continuous, hence by \eqref{maggior} one gets
\[
\| z_m (t) \|_H < 2 \varepsilon
\]
for $m$ large enough and we have proved the lemma.
\end{proof}

\begin{rem}\label{aggiunto}
The thesis of previous lemma remains true if $E_m(t,s)$ is replaced with the adjoint $E_m(t,s)^*$ and $E(t,s)$ is replaced by $E(t,s)^*.$ Indeed, the following formula holds for the adjoint (see also \cite{Dan,Ella}):
\[
 E(t,s)^* x = E^r (T-s,T-t) x \text{ for any fixed }x\in H \text{ and } (t,s)\in \Delta,
\]
where $ \{ E^r (t,s)  \} $ is the evolution system associated to the operator $A^r,$ which in turns is generated by the form $a^r(t,u,v)=\overline{a(T-t,v,u)},$ which also satisfies properties $(H1)-(H4).$ \\
Note that, for a fixed $m\in\mathbb{N},$
\begin{eqnarray*}
&& (a_m)^r (t,u,v) = \overline{a_m (T-t, v,u)} \\
&& = \overline {a (T-t,\P v,\P u)}  + \alpha \overline{\langle (I-\P) v, (I-\P) u \rangle_V} \\
&& = a^r (t,\P u,\P v) + \alpha \langle (I-\P) u, (I-\P) v \rangle_V.
\end{eqnarray*}
This last fact proves that $E_m(t,s)^* x = E_m^r (T-s,T-t),$ where $E_m^r (t,s)$ is associated to the form $a^r (t,\P u,\P v) + \alpha \langle (I-\P) u, (I-\P) v \rangle_V.$ Lemma \ref{lem:wmconverge} can be then applied to $E^r(t,s)$ and $E_m^r(t,s)$ to get the result.
\end{rem}
\section{Nonlinear superposition operators.}\label{sez4}
 We will make use of the next lemma, which takes inspiration from \cite[Lemma 7.11]{BaTe} (see also \cite{MoTe}). Statement and proof are given in the particular setting of the $L^2$ space of $H-$valued functions, though further extensions to a more general setting might be possible. We recall the following
 \begin{defn}\label{demicont}
  Given two Banach spaces $X$ and $Y,$ a function $F:X\to Y$ is said to be demicontinuos if for any sequence $\{x_n\}_{n\in\mathbb{N}} \subset X$ strongly converging to $x\in X,$ one has that $\{ F(x_n) \}_{n\in\mathbb{N}}$ weakly converges to $F(x)$ or, in other words, that
 \[
 w-\lim_{n\to\infty}F(x_n) = F(x)\text{ whenever }x_n\to x.
 \]
 \end{defn}
 \begin{lem}\label{weaktostrong}
 Suppose that  $f:[0,T]\times H \to H$ satisfies
 \begin{enumerate}[(F1)]
 \item $f(\cdot, x)$ is measurable for any $x\in H;$
 \item $f(t, \cdot)$ is demicontinuous in $H$ for any fixed $t\in [0,T].$
 \item there exist $a>0$ and $b \in  L^2 ([0,T],\mathbb{R}_+)$ such that
 \begin{equation}\label{superposcond}
 \| f(t,x) \|_H \leq a \| x \|_H + b(t).
 \end{equation}
\end{enumerate}
 Then the superposition operator $N_f : L^2 ([0,T],H) \to L^2 ([0,T], H)$ given by $N_f (u) (t):=f(t,u(t))$ is well-defined  and maps bounded sets into bounded sets; moreover it is demicontinuous.
 \end{lem}
 \begin{proof}
 The fact that $N_f$ is well-defined and maps bounded sets into bounded sets is given by Theorem \ref{LupA}. It remains to prove that the superposition operator is demicontinuous.\\
 Hence, let $\{u_m\}\subset L^2 ([0,T],H)$ be a sequence which converges strongly to $u,$ then for a.e. $t\in [0,T]$ it holds $u_m(t) \to u(t)$ in $H$  and $f(t,u_m(t))\wto f(t,u(t))$ follows by demicontinuity assumption on $f.$\\
On the hand, we observe that $\{N_f (u_m) \}$ is bounded in $\ldue{H}$ and, by reflexivity, we can assume it converges weakly to $v\in \ldue{H}$ up to subsequences; thus it remains to prove that $v(t)=f(t,u(t)).$\\
 Fix a countable and dense sequence $\{ e_j \}$  in $H;$ then for any $j \in \mathbb{N}$
 \[
 \langle f(t, u_m(t)), e_j \rangle_H \to \langle f(t, u(t)), e_j \rangle_H \text { as }m\to\infty
 \]
for a.e. $t\in [0,T]$ and by Egorov Theorem there exists a null set $A_j$ such that the convergence is uniform on $[0,T]\setminus A_j.$ Since $\bigcup_{j\in\mathbb{N}} A_j$ is also null, it holds for any $j\in\mathbb{N}$ that
\begin{equation}\label{nfconv}
\int_0^T | \langle f(t,u_m (t)) - f(t,u(t)), e_j \rangle_H | dt \to 0 \text { as } m\to\infty.
\end{equation}
Recalling the fact that $\{N_f (u_m) \}$ weakly converges to $v$ in $\ldue{H},$ one in particular gets
\[
\lim_{m\to\infty} \int_0^T | \langle N_f(u_m)(t) - v(t), e_j \rangle_H | dt = 0.
\]
The latter together with \eqref{nfconv} bring to $\langle f(t,u(t))-v(t), e_j \rangle_H = 0$ for any $j\in \mathbb N$ and by density $v=N_f(u)$ follows.
\end{proof}
It is worth noting that the class of demicontinuous function is wide and also includes maximally monotone operators, as we point out in Section \ref{sez6}.

\section{Existence and regularity of solutions}\label{sez5}
This section is devoted to the study on the existence of solution in $H^1 ([0,T],H) \cap \ldue{V}$ of the problem
 \begin{eqnarray}\label{PMR}
	\begin{cases}
		u'(t)+A(t)u(t)=f(t,u(t)) \qquad t\in[0,T]\ \\
		u(0)=g(u),
	\end{cases}
\end{eqnarray}
where,
\begin{enumerate}[(i)]
\item $V$ and $H$ are Hilbert spaces, with $V$ densely embedded into $H$ and the embedding $V\hookrightarrow H$ is compact.
\item $\{A(t): t\in [0,T]\}$ is generated by a sesquilinear form $a$ which satisfies (H1)-(H4) and  (S).
\item $f:[0,T]\times H \to H$ satisfies conditions (F1)-(F3) in Lemma \ref{weaktostrong} and it holds the following
\begin{itemize}
	\item[(T) ]There exists two real numbers $R_{0}>r_{0}>0$ such that for any
	$x\in H$ with $\|x\|_H \in(r_{0},R_{0})$  
	\begin{equation}
		\Re ( \langle f(t,x),x \rangle_H ) \leq0\qquad\forall t\in[0,T]. \label{eq:trasversality}
	\end{equation}
\end{itemize}
\item $g: L^2([0,T],H)\to V$ is demicontinuous, maps bounded sets into bounded sets  and the condition
\begin{equation}\label{gcond}
\| g(u) \|_H < r  \text{ whenever }\frac{\| u \|_{\ldue{H}}}{\sqrt{T}} =:r  \in (r_0  , R_0 )
\end{equation}
holds, where $r_0$ and $R_0$ are the constants from condition (T).
\end{enumerate}

\begin{thm}\label{main}
Suppose that the above conditions hold, then Problem \eqref{PMR} has a solution  $u_* \in H^1 ([0,T],H) \cap \ldue{V};$ moreover the estimate
\begin{eqnarray}\label{stimasol}
\nonumber&&\| u_* \|_{H^1 ([0,T],H)} + \| u_* \|_{\ldue{V}} + \| A(\cdot) u_*(\cdot) \|_{\ldue{H}}\\
&&\leq C_1  \left ( 2\max\{ar_0\sqrt{T}, \|b\|_{\ldue{\mathbb{R}_+}}\} + g^* \right )
\end{eqnarray}
holds, where $g^* :=\sup\{\|g(u)\|_V:\|u\|_{\ldue{H}}\leq r_0\sqrt{T}\}$ and $C_1>0$ is the constant from \eqref{stimamaxreg2}.
\end{thm}
\begin{proof}
 Let $m\in \mathbb{N}$  and $R\in (r_0\sqrt{T},R_0\sqrt{T});$ we will denote by $\bar{B}_{R,m}\subset L^2([0,T],H_m)$ the set
$$
 \bar{B}_{R,m}:=\left \{u\in L^2([0,T],H_m): \|u(\cdot)\|_{\ldue{H}}\leq R\right\}
$$
and let  $A_m(t)$ be the operator given in \eqref{id}. \\
For a fixed $w \in \palla,$ we will at first consider the following problem
\begin{eqnarray}\label{Pdimfin}
		\begin{cases}
			u'(t)+A_m(t)u(t)= \P f(t,w(t)), \qquad a.e. \ \ t\in[0,T]\ \\
			u(0)=\mathbb{P}_mg(w).
		\end{cases}
\end{eqnarray}
From Lemmas \ref{evolsystem}-(iv), \ref{Am} and \ref{weaktostrong} that \eqref{Pdimfin} has a unique solution in $H^1([0,T],H_m)\cap\ldue{V}\cap C([0,T],H_m)$ which can be represented by
\[
	u_w(t):=E_m(t,0)\mathbb{P}_m g(w) +\int_0^t E_m(t,s) \P N_f(w(s)) ds \qquad t\in[0,T],
\]
where $\{E_m(t,s)\}_{(t,s)\in \Delta}$ is the evolution system associated to $A_m(t)$ and $N_f$ is the superposition operator $N_f (w)(t):=f(t,w(t)).$ Also note that the estimate
  \[
 \| u_w(t)\|_{H^1([0,T],H)} \leq C ( \| \P g(w) \|_V + \| N_f w \|_{\ldue{H}})
 \]
holds true for some constant $C>0$ (see Theorem \ref{cor3}). \\
 The next step is to define the map $S:[0,1]\times \palla \to \ldue{H_m}$ by
 \begin{equation}\label{S}
 S(\lambda,w(t))= \lambda E_m(t,0)\mathbb{P}_m g(w) +\int_0^t \lambda E_m(t,s) \P N_f(w(s)) \qquad t\in[0,T].
 \end{equation}
and to prove that $S$ satisfies the hypotheses of Theorem \ref{LeSc}.\\

Firstly observe that $S(0,x)=0\in\operatorname{int}(\palla);$ about the complete continuity of $S$, firstly let $\{(\lambda_k,w_k)\}_{k\in\mathbb{N}}\subset [0,1]\times\bar{B}_{R,m}$ be a sequence such that $\lambda_k\to \lambda_0$ and $w_k\to w_0.$ \\
We show that $S(\lambda_k,w_k) \to S(\lambda_0,w_0)$ in $\ldue{H}$, as $k\to \infty$.\\
	It is not difficult to note that, for $t\in[0,T]$,
	\begin{eqnarray*}\label{diseq}
		&&\|\lambda_k E_m(t,0)\P g(w_k)-\lambda_0E_m(t,0)\P g(w_0)\|_{H}\\
		&\leq&|\lambda_k-\lambda_0|\|E_m(t,0)\P g(w_0)\|_{H}+\lambda_k\|E_m(t,0)\P (g(w_k)-g(w_0))\|_{H}.
	\end{eqnarray*}
The latter implies that
\begin{eqnarray*}
	&& \|\lambda_k E_m(\cdot,0)\P g(w_k)-\lambda_0E_m(\cdot,0)\P g(w_0)\|_{\ldue{H_m}} \\
	& \leq & C( |\lambda_k-\lambda_0| \| g(w_0)\|_H + \| \P g(w_k) - \P g(w_0)\|_H ).
\end{eqnarray*}
	Since $\P$ is weak-to-strong continuous, while $g$ is demicontinuous, we derive that $\P g (w_{k}) \to \P g (w_{0}) $, as $k\to \infty.$
	We have then proved that
\begin{equation}\label{cont1}
	\lim_{k\to\infty} \lambda_k E_m(\cdot,0)\P g(w_k) = \lambda_0 E_m(\cdot,0)\P g(w_0) \text{ in } \ldue{H_m}.
\end{equation}
On the other hand, since $\P$ is weak-to-strong continuous, while $N_f$ is demicontinuous by Lemma \ref{weaktostrong}, we derive that $\P N_f (w_{k}) (s) \to \P N_f (w_{0}) (s)$, as $k\to \infty.$
Note that $w_k (s) \to w_0 (s)$ a.e. uniformly on $[0,T]$ by Egorov theorem, so that $\|w_k (s) \|_H \leq \|w_0 (s) \|_H+\varepsilon_0$ holds uniformly on $[0,T]$ for some $\varepsilon_0 >0.$ From condition (F3) and the contractivity of the evolution system, it is derived that
\begin{eqnarray}\label{diseq2}
\nonumber	&& \|E_m(t,s)\P N_f (w_{k}) (s)\|_{H} \\
		&\leq &  a  \|w_{0} (s) \|_{H} + \varepsilon_0 +b (s) \text { for a.e. }t\in [0,T],
\end{eqnarray}
with $$\Big\| a \|w_0 (\cdot)  \|_{H} + \varepsilon_0 + b(\cdot)\Big\|_{L^1([0,T],\mathbb{R}_+)} \leq C (  R+\varepsilon_0 +\|b\|_{\ldue{\mathbb{R}_+}} ).$$
Then, by Lebesgue Dominated Convergence Theorem,
	$$
	\int_0^t E_m(t,s)\mathbb{P}_mN_f (w_{k}) (s)ds\to \int_0^t E_m(t,s)\mathbb{P}_mN_f (w_{0}) (s)ds
	$$
a.e. on $t\in [0,T]$ as $k\to \infty;$ since $\lambda_k\to \lambda_0$ too, then
$$
	\lim_{k\to\infty} \lambda_k \int_0^t E_m(t,s)\mathbb{P}_mN_f (w_{k}) (s)ds  = \lambda_0  \int_0^t E_m(t,s)\mathbb{P}_mN_f (w_{0}) (s)ds
$$
for a.e. $t\in[0,T].$ Similarly as in \eqref{diseq2} it can be proved that the above convergence is dominated by the same function in $\ldue{H_m}.$ Hence,
$$
\int_0^{(\cdot)} E_m(\cdot,s)\mathbb{P}_mN_f (w_{k}) (s)ds\to \int_0^{(\cdot)} E_m(\cdot,s)\mathbb{P}_mN_f (w_{0}) (s)ds
$$
in $\ldue{H}$ as $k\to \infty$ and so
\[
\lim_{k\to\infty} S(\lambda_k,w_k) = S(\lambda_0,w_0) \text{ in }\ldue{H_m}
\]
and the continuity of the operator $S$ is proved.\\
We next show that $S$ maps bounded sets into compact ones. To this end, it is enough to show that $S([0,1],\bar{B}_{R,m})$ is relatively compact in $L^2([0,T],H_m)$.
Observe that, whenever $u\in S([0,1],\bar{B}_{R,m}),$ then it is the unique solution of
\begin{eqnarray}\label{p2}
		\begin{cases}
			u'(t)+A_m(t)u(t)= \P \lambda f(t,w(t)), \qquad a.e. \ \ t\in[0,T]\ \\
			u(0)=\mathbb{P}_m \lambda g(w)
		\end{cases}
\end{eqnarray}
for some $\lambda \in [0,1]$ and $w\in \palla.$ By Lemma  \ref{Am}, (F3) and Theorem \ref{cor3}, it is true that $u\in H^1([0,T],H_m)$ and that
\begin{eqnarray}\label{limitatH1}
\nonumber \| u \|_{H^1([0,T],H_m)} & \leq & C_1( \| \P g(w) \|_V + a \| w \|_{\ldue{V_m}} + \|b\|_{\ldue{\mathbb{R}_+}} ) \\
\nonumber & \leq & C_1 (  \sup\{ \| g(w)\|_V:\|w\|_{\ldue{H}}\leq R\} + 1+ a R + \| b \|_{\ldue{\mathbb{R}_+}} ) \\
&< &\infty
\end{eqnarray}
since $g$ maps bounded sets into bounded sets and where we supposed for simplicity that $\| g(u)-\P g(u)\|_V<1.$ \\ This last fact shows that $S([0,1],\bar{B}_{R,m})\subset H^1([0,T],H_m)$ is uniformly bounded and hence relatively compact in $L^2([0,T],H_m)$, by Theorem \ref{aubinlions}.\\
In order to apply Theorem \ref{LeSc} it remains to prove that the set $\{ u\in \palla: S(\lambda, u)=u \text { for some } \lambda\in [0,1]  \}$ is bounded and has no intersection with the boundary of $\palla.$ To this end, fix $\lambda \in (0,1)$ and suppose that $\tilde{u}=S(\lambda , \tilde{u})$ for some $\tilde{u}$ with $\| \tilde{u} \|_{\ldue{H_m}}=R\in(r_0 \sqrt{T}, R_0 \sqrt{T}).$ Note that \eqref{gcond} implies that \[
\| \tilde{u}(0) \|_H^2\leq \|g(\tilde {u})\|^2_H< \frac{R^2}{{T}} = \frac{1}{{T}} \| \tilde{u} \|^2_{\ldue{H_m}}.
\]
Since $\tilde{u}$ is continuous and by the mean value theorem, two points $t_1, t_2 \in (0,T]$ must exist such that
\begin{equation}\label{diseq3}
 r_0^2 < \|\tilde{u}(t_1) \|^2_{H} < \|\tilde{u}(t_2 ) \|^2_{H}= \frac{1}{{T}} \|\tilde{u}\|^2_{\ldue{H_m}}<R_0^2.
 \end{equation}
Note that condition (T) implies that for any $t\in [t_1,t_2]$
\begin{equation}\label{condt}
 \langle \P N_f (\tilde{u}) (t), \tilde{u}(t) \rangle_{H}   = \langle f(t,\tilde{u}(t)), \P \tilde{u}(t) \rangle_{H}  = \langle f(t,\tilde{u}(t)), \tilde{u}(t) \rangle_H\leq 0.
\end{equation}
Since $A_m$ is accretive, \eqref{diseq3} holds and by \eqref{condt}, the contradiction
\begin{eqnarray}\label{discentr}
\nonumber 0 & < & \frac { \| \tilde{u} (t_2) \|_{H}^2 - \| \tilde{u} (t_1) \|_{H}^2 }{2}  =\int_{t_1}^{t_2} \frac{d}{dt} \left(\frac {1}{2} \| \tilde{u}(t) \|^2\right) dt\\
\nonumber & = &  \int_{t_1}^{t_2}  \Re (\langle \tilde{u}'(t), \tilde{u}(t) \rangle_H) dt  =  \int_{t_1}^{t_2} - \Re (\langle A_m(t)\tilde{u}(t), \tilde{u}(t) \rangle_H)dt \\
&& +\lambda \int_{t_1}^{t_2}  \Re ( \langle \P N_f (\tilde{u}) (t), \tilde{u}(t) \rangle_H )dt \leq 0
\end{eqnarray}
arises. We have then proved that $\| \tilde{u} \|_{\ldue{H_m}}<R$  and then $S(1,\cdot)$ has a fixed point by Theorem \ref{LeSc}, that is there exists a solution $u_m$ to the problem
\begin{eqnarray}\label{Pm}
		\begin{cases}
			u'(t)+A_m(t)u(t)= \P f(t,u(t)), \qquad a.e. \ \ t\in[0,T]\ \\
			u(0)=\mathbb{P}_mg(u);
		\end{cases}
\end{eqnarray}
moreover, $u_m \in H^1([0,T],H_m)\cap\ldue{V}\cap C([0,T],H).$\\
Consider the sequence $\{ u_m \}_{m\in\mathbb{N}} \subset \ldue{H}$ and note that it is bounded in $H^1([0,T],H)$ since the bound in \eqref{limitatH1} does not depend on $m\in\mathbb{N}.$ Then by Theorem \ref{aubinlions} a subsequence $\{u_{m_k}\}_{k\in\mathbb{N}}$ exists which converges to a point  $u_*\in \ldue{H}.$ We may also assume  that $u_{m_k}(t)\to u_*(t)$ as $k\to \infty$ for a.e. $t\in [0,T].$\\

By passing to a further subsequence, it can be seen that for any fixed $t\in[0,T]$
\begin{equation}\label{conv1}
\int_0^t E_{m_k}(t,s)\mathbb{P}_{m_k} f(s,u_{m_k}(s)) - E(t,s) f(s,u_*(s))ds \to 0.
\end{equation}
Indeed, for a fixed $t\in [0,T],$ consider \[ y_{m_k}(t):= \int_0^t E_{m_k}(t,s)\mathbb{P}_{m_k} f(s,u_{m_k}(s))-  E(t,s)f(s,u_*(s))  ds \] and note that $\{ y_{m_k} \}_{k\in\mathbb{N}} $ can be seen as a bounded sequence in $H^1 ([0,T],H)$ by Lemma \ref{evolsystem}-(iv) and since $H^1 ([0,T],H_{m_k})\hookrightarrow H^1 ([0,T],H).$ Reasoning as before and passing to further subsequences if necessary, it is readily proved that $y_{m_k}(t)\to y_*(t)$ as $k\to\infty.$ \\
On the other hand, fix $x\in H$ and $(t,s)\in \Delta$  note that,
\begin{eqnarray}\label{conv2}
\nonumber &&\langle E_{m_k}(t,s)\mathbb{P}_{m_k} f(s,u_{m_k}(s)) - E(t,s) f(s,u_*(s)), x \rangle_H \\
\nonumber &&=\langle  E_{m_k}(t,s)\mathbb{P}_{m_k} ( f(s,u_{m_k}(s)) - f(s,u_*(s))), x \rangle_H\\
\nonumber &&\phantom{=} + \langle (  E_{m_k}(t,s)\mathbb{P}_{m_k} - E(t,s) )f(s,u_*(s)),x \rangle_H\\
\nonumber &&=\langle  f(s,u_{m_k}(s)) - f(s,u_*(s)), E(t,s)^* x \rangle_H \\
\nonumber && \phantom{=} + \langle   f(s,u_{m_k}(s)) - f(s,u_*(s)) ,(E_{m_k}(t,s)\mathbb{P}_{m_k})^* x  - E(t,s)^* x \rangle_H\\
&&\phantom{=} + \langle (  E_{m_k}(t,s)\mathbb{P}_{m_k} - E(t,s) )f(s,u_*(s)),x \rangle_H.
\end{eqnarray}
Let $k\to\infty$ and observe that
 $$\langle  f(s,u_{m_k}(s)) - f(s,u_*(s)), E(t,s)^* x \rangle_H \to 0$$ since $f(s,\cdot)$ is demicontinuous, while  $$\langle   f(s,u_{m_k}(s)) - f(s,u_*(s)) ,(E_{m_k}(t,s)\mathbb{P}_{m_k})^* x  - E(t,s)^* x \rangle_H\to 0$$ by Remark \ref{aggiunto} and since $ f(s,u_{m_k}(s)) - f(s,u_*(s))$ is bounded; lastly, $$\langle (  E_{m_k}(t,s)\mathbb{P}_{m_k} - E(t,s) )f(s,u_*(s)),x \rangle_H\to 0$$ by Lemma \ref{lem:wmconverge}.\\
 Then \eqref{conv2} implies that \[ \lim_{k\to\infty} \langle E_{m_k}(t,s)\mathbb{P}_{m_k} f(s,u_{m_k}(s)) - E(t,s) f(s,u_*(s)), x \rangle_H = 0.\]
Also, it is easily seen that for any fixed $t\in[0,T]$ and any $s\in [0,t],$
\begin{eqnarray*}
 | \langle E_{m_k}(t,s)\mathbb{P}_{m_k} f(s,u_{m_k}(s)) - E(t,s) f(s,u_*(s)), x \rangle_H | & \leq & \|x\|+(1+a^2)\| x\|^2\\
 && +a^2 \|u(s)\|_H^2+|b(s)|^2
\end{eqnarray*}
where we have supposed that $\| u_{m_k}(s) \|_H \leq 1+\|u_*(s)\|_H$ a.e. uniformly on $[0,t]$ and for $k$ large enough. Since the last term of the inequality lies in $L^1([0,t],H),$ the Lebesgue's Dominated Convergence Theorem implies that
\begin{eqnarray*}
&& \langle y_{m_k} (t), x \rangle_H =\left \langle \int_0^t E_{m_k}(t,s)\mathbb{P}_{m_k} f(s,u_{m_k}(s))-  E(t,s)f(s,u_*(s))   ds , x \right \rangle_H\\
&&= \int_0^t   \langle E_{m_k}(t,s)\mathbb{P}_{m_k} f(s,u_{m_k}(s))-  E(t,s)f(s,u_*(s)),x\rangle_H ds \to 0 \text{ as }k\to\infty,
\end{eqnarray*}
where we have used \cite[Proposition 23.9]{Ze}.\\
By uniqueness of the weak limit and since $x\in H$ is arbitrary, it follows that $y_* = 0$ and hence \eqref{conv1} is proved.\\

Further we note that
\begin{eqnarray}\label{conv3}
	\nonumber  \| E_{m_k}(t,0) \mathbb{P}_{m_k} g(u_{m_k}) - E(t,0)g(u_*)\|_H & \leq & \| E_{m_k}(t,0) \mathbb{P}_{m_k}(g(u_{m_k})-g(u_*))\|_H \\
	\nonumber && + \| (E_{m_k}\mathbb{P}_{m_k} (t,0) -  E(t,0))g(u_*)\|_H
\end{eqnarray}
and by Lemma \ref{lem:wmconverge}, $\| (E_{m_k} (t,0) \mathbb{P}_{m_k}-  E(t,0))g(u_*)\|_H\to 0$ as $k\to \infty$. In order to prove that $\| E_{m_k}(t,0) \mathbb{P}_{m_k}(g(u_{m_k})-g(u_*))\|_H\to 0$, as above let us introduce
$$
z_{m_k}(t):= E_{m_k}(t,0) \mathbb{P}_{m_k}(g(u_{m_k})-g(u_*))
$$
that is a bounded sequence in $H^1 ([0,T],H)$. Passing to further subsequences if necessary, it is readily proved that $z_{m_k}(t)\to z_*(t)$ as $k\to\infty.$
On the other hand, fix $x\in H$ and $(t,s)\in \Delta$  and note that,
\begin{eqnarray*}
	\nonumber\left \langle z_{m_k}, x \right \rangle_H & =& \langle  g(u_{m_k})-g(u_*), (E_{m_k}(t,0) \mathbb{P}_{m_k})^* x \rangle_H \\
	\nonumber &=& \langle    g(u_{m_k})-g(u_*),(E_{m_k}(t,0)\mathbb{P}_{m_k})^* x- E(t,0)^* x \rangle_H\\
	&&+ \langle  g(u_{m_k})-g(u_*),E(t,0)^*x \rangle_H.
\end{eqnarray*}
When $k\to\infty$ we observe that $\langle  g(u_{m_k})-g(u_*),E(t,0)^*x \rangle_H \to 0$ since $g(\cdot)$ is demicontinuous, while $$\langle g(u_{m_k})-g(u_*),(E_{m_k}(t,0)\mathbb{P}_{m_k})^* x- E(t,0)^* x \rangle_H\to 0$$ by Remark \ref{aggiunto} and since $g(u_{m_k})-g(u_*)$ is bounded. This proves that $\langle z_{m_k}(t),x\rangle_H \to 0$, as $k\to \infty$ and so $z_*=0$.

The fact that $g(u_{m_k})\to g(u_*)$ as $k\to \infty$ follows by arguing as before.
Putting all together, we have then proved that for any $t\in[0,T],$
\begin{eqnarray*}
u_* (t) & = & \lim_{k\to\infty} u_{m_k} (t)\\
&=& \lim_{k\to\infty}  E_{m_k}(t,0)\mathbb{P}_{m_k} g(u_{m_k}) + \int_0^t E_{m_k}(t,s)\P f(s,u_{m_k}(s))ds \\
&=& E(t,0)g(u_*) +  \int_0^t E(t,s)f(s,u_*(s))ds.
\end{eqnarray*}
Hence, by Lemma \ref{evolsystem}, $u_*\in H^1 ([0,T],H) \cap \ldue{V} $ solves Problem \eqref{PMR}. Let $g^*_R :=\sup\{\|g(u)\|_V:\|u\|_{\ldue{H}}\leq R \sqrt{T}\}$; it holds
\begin{eqnarray*}
&&\| u_* \|_{H^1 ([0,T],H)} + \| u_* \|_{\ldue{V}} + \| A(\cdot) u_*(\cdot) \|_{\ldue{H}} \\
&&\leq C_1 \left ( \| N_f (u_*) \|_{\ldue{H}}  + \| g(u_*) \|_V \right ) \\
&&\leq C_1 \left ( 2\max\{a\|u_*\|_{\ldue{H}}, \|b\|_{\ldue{\mathbb{R}_+}}\} + g^*_R \right )\\
&&\leq C_1  \left ( 2\max\{aR, \|b\|_{\ldue{\mathbb{R}_+}}\} + g^*_R \right ).
\end{eqnarray*}
To reach the estimate \eqref{stimasol}, we observe that in the latter inequality, $R$ can be chosen to be arbitrarily close to $r_0.$
\end{proof}
\begin{cor}\label{maincor}
Let $\{A(t): t\in [0,T]\}$ be generated by a sesquilinear form $a(\cdot,\cdot,\cdot)$ which satisifies (H1)-(H3) and (H4$^*$). Moreover, suppose that (S) holds and that $f[0,T]\times H \to H$ satisfies conditions (F1)-(F3) in Lemma \ref{weaktostrong} for some $b\in L^\infty ([0,T],\mathbb{R}_+)$.
Suppose also that $g$ is demicontinuous and such that $ g_*:=\sup \{ \|g(u)\|_V : u\in \ldue{H}\}<+\infty$ holds.
If $V$ is compactly embedded in $H$ then problem
 \begin{eqnarray}\label{PC}
		\begin{cases}
			u'(t)+A(t)u(t)= f(t,u(t)), \qquad a.e. \ \ t\in[0,T]\ \\
			u(0)= g(u)
		\end{cases}
\end{eqnarray}
admits a solution in $H^1 ([0,T],H) \cap \ldue{V}$
\end{cor}
\begin{proof}
We have already seen in Remark \ref{trucco} that by setting $v(t):=e^{-\mu t} u(t)$, problem \eqref{PC} is equivalently rewritten as
\begin{eqnarray}\label{PC2}
		\begin{cases}
			v'(t)+(A(t)+\mu I)v(t)= e^{-\mu t} f(t,e^{\mu t}v(t)), \qquad a.e. \ \ t\in[0,T]\ \\
			v(0)= g ( e^{\mu \cdot} v(\cdot) ).
		\end{cases}
\end{eqnarray}
We will see that the same trick permits us to weaken the hypothesis on the nonlinear term. Indeed, choose $\mu := \delta + \epsilon,$ where $\epsilon > a$, being $a$ the constant from (F3). We rewrite \eqref{PC2} as
\begin{eqnarray}\label{PC3}
		\begin{cases}
			v'(t)+\hat A (t)v(t)= \hat{f}(t,v(t)), \qquad a.e. \ \ t\in[0,T]\ \\
			v(0)= \hat{g}(v),
		\end{cases}
\end{eqnarray}
where $\hat A(t) = A(t) + \delta I$ is accretive, $\hat{g}:v\in \ldue{H} \mapsto g( e^{\mu \cdot} v(\cdot))\in V$ is demicontinuous and bounded and $\hat{f}(t,x):=e^{-\mu t} f(t,e^{\mu t}x) - \epsilon x$ still satisfies properties (F1)-(F3).\\
To see that $\hat{f}$ also satisfies the transversality condition (T), note that for any $(t,x)\in [0,T]\times H$
\begin{eqnarray*}
&& \Re \langle \hat{f} (t,x),x \rangle_H =\Re \langle e^{-\mu t} f(t,e^{\mu t}x),x \rangle_H -\epsilon \| x\|^2 \\
&& \leq  a \| x \|^2 + e^{-\mu t}  b(t) \|x\| -\epsilon \| x\|^2\\
&& \leq \left ( (a-\epsilon)\|x\| + \| b \|_\infty \right ) \|x\|.
\end{eqnarray*}
Since $\epsilon > a,$ for any $\|x\| > \| b \|_\infty/(\epsilon - a)$ it holds
\[
	 \Re \langle\hat{f} (t,x),x \rangle_H \leq 0.
\]
In particular, condition (T) is satisfied for any $r_0> \frac{\|b\|_\infty}{\epsilon -a}.$\\
Set $r_0 > \max \left \{ \sqrt{T}g_*, \frac{\|b\|_\infty}{\epsilon -a }\right\}$ and $R_0 := +\infty,$ then $\hat{g}$ satisfies the hypotheses  of previous result, which can be followed closely to prove the existence of a solution to \eqref{PC3}.
\end{proof}

\section{Applications.}\label{sez6}
Models of non-autonomous operators $A(t)$ governed by forms satisfying assumptions (H1)-(H4) and (S) had been widely investigated in literature (see e.g. \cite{ACFP,HaOu}). Here we briefly introduce the topic.\\
Let $\Omega$ be an open connected and bounded subset of $\mathbb{R}^n$ whose boundary is a finite union of parts of rotated graphs of Lipschitz maps, i.e. a strongly Lipschitz domain. Let $H$ be $L^2(\Omega)$ with the classical Lebesgue measure, while with $H^1(\Omega)$ we denote the classical Sobolev space consisting of all square-integrable functions $u:\Omega\to \mathbb{R}$ such that $\nabla u$ exists in the weak sense and belongs to $L^2(\Omega).$ We denote by  $H_0^1(\Omega)$ the closure of $C_0^\infty(\Omega)$ in $H^1(\Omega).$  Since $\Omega$ is open and bounded, by using Poincar\'e inequality, we equip $H_0^1(\Omega)$ with the equivalent norm $\|u\|_{H_0^1(\Omega)}:=\|\nabla u\|_{L_2(\Omega)}$. In particular, by Rellich-Kondrakov's Theorem, $H_0^1 (\Omega)\hookrightarrow L^2(\Omega)$ compactly.\\
Note that $H$ is reflexive and separable, so that $L^2([0,T],H)$ is also reflexive and separable; moreover $H^1([0,T],H)$ is compactly embedded on $L^2([0,T],H)$.\\
We  recall the following result proved \cite{AuTc}, which provides a wide class of operators satisfying condition (S).
\begin{thm}\label{kspdiv} Let $\Omega\subset \mathbb{R}^n$ be a strongly Lipschitz domain and $A$ a bounded uniformly elliptic complex matrix on $\Omega$. Let $L:=-\operatorname{div}(A\nabla),$ with $L:V \to H^1(\Omega)$ and $V=H^1_0(\Omega)$ or $V=H^1(\Omega) .$ Then the domain of the maximal accretive square root $L^\frac{1}{2}$ agree with $V$ with equivalence of norms.
\end{thm}\ \\
To provide a model for sesquilinear forms we set $V:=H^1(\Omega)$ and define 
\begin{equation}\label{formsesq}
a(t,u,v) :=\int_\Omega\sum_{i,j=1}^{n} a_{i,j}(t,x)\partial_iu \ \overline{\partial_jv} \ \d x,
\end{equation}
then $a:[0,T]\times V \times V \to \mathbb{C}.$
The assumptions on $a$ are the following:
\begin{description}
	\item [{(a)}] Let $a_{i,j}(\cdot,\cdot)\in L^\infty([0,T]\times \Omega,\mathbb{R})$ ($i,j=1,\dots,n$)
   satisfying:
	\begin{description}
		\item [{(a1)}] the (uniform) ellipticity condition $$  \sum_{i,j=1}^{n} a_{i,j}(t,x)\xi_i\xi_j  \geq \nu \|\xi\|^2, \qquad \forall \xi\in \mathbb{R}^n, \ a.e.  \ (t,x)\in [0,T]\times \Omega,
		$$
		where $\nu>0$, holds.
		\item [{(a2)}] There exists $K>0$ such that for any $i,j$,
		$$
		|a_{i,j}(t_1,\xi)-a_{i,j}(t_2,\xi)|\leq K |t_1-t_2|^{\alpha}
		$$
		a.e. $\xi \in \Omega$, $t_1,t_2\in[0,T]$ and $\displaystyle \alpha>\frac{1}{2}$
	\end{description}
\end{description}
\begin{prop}\label{Aesempio}
The sesquilinear form $a(t,u,v)$ satisfies assumptions $(H1)-(H4)$ and $(S).$
\end{prop}
\begin{proof}
While $(H1)$ is directly obtained, $(H2)$ and $(H3)$ derives from  the following energy estimate (see \cite{AuTc}) which holds in our setting:
$$
\alpha \|u|^2_{V}\leq a(t,u,u)
$$
for $u\in V$, a.e. in $t\in[0,T]$  and
\begin{equation}\label{qc}
|a(t,u,v)|\leq M \|u\|_V \|v\|_{V}
\end{equation}
for some $M,\alpha>0$ and $u,v\in V.$ At the end, note that $(H4)$ follows for $\omega(t)=Kt^\alpha$ by $(a2).$\\
Following \cite{HaOu}, define the linear operator $\mathcal{A}(t):V\to V'$ such that $\langle \mathcal{A}(t)u,v\rangle_V=a(t,u,v)$, where $\langle \cdot,\cdot\rangle_V$ denotes the usual pairing in $V'\times V$.
For each $t\in [0,T]$, the part of $\mathcal{A}(t)$ on $H$ is given by
$$
A(t):=-\sum_{i,j=1}^{n}\partial_i\left(a_{i,j}(t,\cdot)\partial_j\right)
$$
on $\mathcal{D}(A(t)):=\{u\in V:\mathcal{A}(t)u\in H\}$.\\
Fix $t\in[0,T].$ Thanks to uniform quasi-coercivity and uniform boundedness of $a(t,\cdot,\cdot)$, by Theorem \ref{kspdiv}, the square root property holds, i.e. $\mathcal{D}(A(t))^{\frac{1}{2}})=V.$
\end{proof}
To better illustrate the application range of our main results, in the sequel we analyze two settings in which Theorem \ref{main} applies. The first one regards evolution variational inequality problems.  This problem had been widely investigated in the past with a complete understanding when the problem is governed by a maximal monotone operator (\cite{Bre2}). Neverthless  it represents an interesting test-bed for our result: indeed, we show that hypotheses (F1)-(F3) on the nonlinear term are naturally satisfied by maximal monotone operators with a sublinear growth, so that existence of solutions and regularity properties derive from Theorem \ref{main}. Lastly, we focus our attention on the nonlocal initial condition $u(0)=g(u).$ Nonlocal initial boundary value problems for semilinear equations often arise in concrete phyical model and, in particular, in heat conduction or diffusion processes (\cite{BeCi}). In particular, multipoint initial conditions are used to describe the diffusion phenomenon of a small amount of gas in a transparent tube, where several consecutive measurements are more effective (\cite{Den}).
\subsection{Evolution variational inequality problems.}  
As highlighted before, an important class of evolutionary problems is represented by evolution variational inequalities. Evolution variational inequalities  have been successfully applied in several fields of science with applications to oligopolistic markets, urban transportation networks, traffic networks, international trade, agricultural and energy markets (see e.g. \cite{Du}, \cite{Fr}, \cite{Go}, \cite{Na} and references therein). 
Let $\varphi: L^2(\Omega) \to \mathbb{R}$ be a proper, convex and Gateaux differentiable function with a sublinear gradient growth, i.e. there exists $M,b>0$ such that
\[
\| \nabla \varphi (u) \|_{L^2(\Omega)} \leq M \| u \|_{L^2(\Omega)} + b,
\]
for any $u\in L^2(\Omega).$\\
An example of such functions is  given by the class $C^{1,1}_L$ of the convex and Gateaux differentiable functions with Lipschitz continuous gradient; for more details on the class and its applications to Optimization we refer the reader to \cite{Ne} and references therein.\\

Let $u_0\in H^1(\Omega)$ and $A(t)$ be as in Proposition \ref{Aesempio}; we are interested in the following evolution variational inequality problem
\begin{eqnarray}\label{varin}
\begin{cases}
\langle u'(t) + A(t)u(t), v - u(t) \rangle_{L^2(\Omega)} \geq \varphi(u(t)) - \varphi(v)\text{, for any  } v\in L_2(\Omega)\\
u(0)=u_0.
\end{cases}
\end{eqnarray}
By following \cite[pp. 893-894]{Ze2b}, we rewrite the previous inequality as
\begin{eqnarray}\label{eqvarin}
\begin{cases}
u'(t)+A(t) u(t) = -  \nabla \varphi (u(t)), \\
u(0)=u_0.
\end{cases}
\end{eqnarray}
In the sequel, we prove that the nonlinearity  $f(t,x):=-\nabla \varphi (x)$ satisfies assumptions (F1)-(F3), thus fullfilling the requests.

Observe that $\nabla \varphi$ is maximal monotone and has full domain since $\varphi$ is convex and Gateaux differentiable on $L^2(\Omega),$ which implies that $x\mapsto \nabla \varphi(x)$ is demicontinuous (see \cite[Corollary 21.21]{bauschkecombettesconvex}) and the same property is then satisfied by $f(t,\cdot).$ Condition (F3) immediately derives from the sublinearity of $\nabla \varphi.$
By setting $g(u):=u_0,$ we see that the hypotheses of Corollary \ref{maincor} are fulfilled and the following result holds.
\begin{prop}\label{prop:varin}
Problem \ref{varin} admits a solution in $H^1 ([0,T],L^2(\Omega)) \cap L^2([0,T],H^1(\Omega)).$
\end{prop}
\subsection{Semilinear equations with nonlocal initial conditions.}
As we have already pointed out, several expressions of nonlocal initial conditions had been intensively studied in literature. Among others, we cite \cite{BoPre} where  a multipoint condition had been studied under further compactness assumption,  while in \cite{PaVr}, the authors considered a general integral nonlocal initial condition of type
\[
u(0)=\int_{[0,T]} \mathcal{V}(u(t)) \d{t}
\]
has been introduced. In order to better expose the range of applications provided by Theorem \ref{main}, we will deal with an initial condition of this type.\\ As highlighted before, existing literature on the topic mainly deals with existence of mild solutions to nonlocal initial value problems. Here, we are interested in proving the existence of strong solutions and further regularities; this result can be achieved by introducing a smoothing term, below represented by the convolution with a sufficiently smooth function.\\
Let $\Omega=\mathbb{R}^n$ and let $A(t)$ be associated to the sesquilinear form given \eqref{formsesq}; then  assumptions (H1)-(H4) and (S) are still with $V=H^1 (\mathbb{R}^n)$ (see \cite{HaOu} and \cite{Pas}). 
We are interested in the following nonlocal evolution problem:
\begin{eqnarray}\label{nonlocale}
\begin{cases}
u'(t)+A(t) u(t) = f(t),\ t\in(0,1] \\
\displaystyle u(0)= \int_{I} \varphi*u(t) \d t,
\end{cases}
\end{eqnarray}
where $f\in L^\infty([0,1],\mathbb{R}),\ I=\bigcup [s_i,t_i] \subset [0,1],$ $\varphi \in C^1 (\mathbb{R}^n)$ is a mollifier with $\| \nabla \varphi \|_{L^1(\mathbb{R}^n)} < 1$ and where $\varphi * u(t)$ represents the convolution product among the two terms.\\
Note that by choosing $\varepsilon >0$ small enough, $A_\varepsilon (t):=A(t)-\varepsilon I$ still satisfies the required properties, while $f_\varepsilon (t, x) := f(t)-\varepsilon x$ fullfills requirements $(F1)-(F3)$ and the transversality condition \eqref{eq:trasversality} for any $R_0>r_0 > \| f \|_\infty.$\\
By a standard argument, it is promptly derived that $\displaystyle g(u) = \int_{I} \varphi*u(t) \d t$ sastisfies
\[
\| g(u) {\|^2}_{H^1(\mathbb{R}^n)} \leq \| \nabla \varphi \|_{L^1(\mathbb{R}^n)} \int_{[0,1]} \| u(t) \|^2_{L^2(\mathbb{R}^n)} \d t.
\]
 The latter proves that $g: L^2([0,T],{L^2(\mathbb{R}^n)}) \to H^1(\mathbb{R}^n)$ is well-defined, continuous and satisfies \eqref{gcond}. At this point, an immediate application of Theorem \ref{main} brings the following
 \begin{prop}
 A solution $u_* \in H^1 ([0,T],L^2(\mathbb{R}^n)) \cap \ldue{H^1(\mathbb{R}^n)}$ to problem \eqref{nonlocale} exists for which the a priori estimate \eqref{stimasol} holds.
 \end{prop}

\section*{Acknowledgements}
Both authors wish to thank Professor Wolfgang Arendt for some useful suggestions.

\section*{Data Availability}
Data sharing not applicable to this article as no datasets were generated or analysed during the current study.

\bibliographystyle{elsarticle-harv}

\end{document}